\documentclass[10pt]{article}
\usepackage{amsmath,amssymb}
\usepackage{enumerate}

\topskip=\baselineskip  \headheight=0in
\allowdisplaybreaks[1]

\newtheorem{theorem}{Theorem}[section]

\newtheorem{lemma}[theorem]{Lemma}
\newtheorem{proposition}[theorem]{Proposition}

\newtheorem{remark}[theorem]{Remark}

\newcommand{\Z}{\mathbb{Z}}

\newcommand{\id}{\operatorname{id}}

\newcommand{\aut}{\operatorname{Aut}}

\newcommand{\Aut}{\operatorname{Aut}}

\newenvironment{proof}{\par\noindent{\bf Proof.}}{$\qed$\par\bigskip}
\newcommand{\qed}{\enspace\vrule  height6pt  width4pt  depth2pt}
\usepackage{color}

\begin{document}
\title{Every finite abelian group is a subgroup of the additive group of a
finite simple left brace\thanks{The first author was partially
supported by the grants MINECO-FEDER  MTM2017-83487-P and AGAUR
2017SGR1725 (Spain). The second author is supported in part by
Onderzoeksraad of Vrije Universiteit Brussel and Fonds voor
Wetenschappelijk Onderzoek (Belgium). The third author is supported
by the National Science Centre  grant. 2016/23/B/ST1/01045 (Poland).
2010 MSC: Primary 16T25, 20F16, 20E22. Keywords: Yang-Baxter
equation, set-theoretic solution, brace, simple, A-group.} }
\author{F. Ced\'o \and E. Jespers \and J. Okni\'{n}ski}
\date{}

\maketitle

\begin{abstract}
Left braces, introduced by Rump, have turned out to provide an
important tool in the study of set theoretic solutions of the
quantum Yang-Baxter equation. In particular, they have allowed to
construct several new families of solutions.  A left brace
$(B,+,\cdot )$ is a structure determined by two group structures on
a set $B$: an abelian group $(B,+)$ and a group $(B,\cdot)$,
satisfying certain compatibility conditions. The main result of this
paper shows that every finite abelian group $A$ is a subgroup of the
additive group of a finite simple left brace $B$ with metabelian
multiplicative group with abelian Sylow subgroups. This result
complements earlier unexpected results of the authors on an
abundance of finite simple left braces.
\end{abstract}

\section{Introduction}

In order to investigate a question posed by Drinfeld
\cite{Drinfeld}, Rump, in  \cite{R07}, introduced a new algebraic
structure, called a left brace, to study non-degenerate involutive
set-theoretic solutions of the Yang--Baxter equation. Using an
equivalent formulation, given in \cite{CJOComm}, a left brace is a
set $B$ equipped with two operations $+$ and $\cdot$ such that
$(B,+)$ is an abelian group, $(B,\cdot )$ is a group and $a\cdot
(b+c)+a=a\cdot b + a\cdot c$, for all $a,b,c\in B$. In \cite{BCJAll}
it was shown that all non-degenerate involutive solutions on a
finite set can be explicitly determined from finite left braces.
Hence the  study of left braces becomes essential to classify all
solutions. Moreover, intriguingly, braces have shown up in several
areas of mathematics (see for example the surveys
\cite{CedoSurvey,RumpSurvey}). One of the fundamental problems is to
classify the building blocks of all finite simple left braces, that
is describe all finite simple left braces. It is well-known that the
additive Sylow $p$-subgroup $B_p$  of a finite brace $B$ is a left
subbrace of  $B$ and, in \cite{BCJO18}, it has been shown that $B$
is an iterated matched product of all $B_p$. Rump, in \cite{R07},
has shown that if $B=B_p$, then $B$ is a simple left brace precisely
when $B$ has order $p$; more generally, the conclusion holds if
$(B,\cdot )$ is a finite nilpotent group. In particular, such a
brace is trivial, i.e. the two operations $+$ and $\cdot$ coincide.
Recall that Etingof, Schedler and Soloviev \cite{ESS} have shown
that if $B$ is finite, then $(B,\cdot )$ is a solvable group. Also
if $B_p\neq B$, there are some natural constraints on the order of a
finite simple left brace $B$. For example, Smoktunowicz, in
\cite{SmokRestraints}, showed that if $|B|=p^nq^m$ (with $p$ and $q$
different prime numbers and $n$, $m$ positive integers), then
$p|(q^t - 1)$ and $q|(p^s - 1)$ for some $0 < t \leq  m$ and $0 < s
\leq  n$. However,   these conditions are not sufficient for
simplicity of $B$. For example, there is no simple left brace of
order $p^n q$, where $p$ and $q$ are distinct primes and $n$ is the
multiplicative order of $p$ in the unit group of $\Z/(q)$ (see
\cite[Remark~5.3]{BCJO18}). Bachiller, in \cite[Theorem 6.3 and
Section 7]{B18}, produced the first example of a non-trivial finite
simple left brace. This initiated a program of constructing and
describing finite simple left braces \cite{BCJO18,BCJO19}. One of
the structural approaches is via the matched product of the Sylow
subgroups; this allowed to construct new classes of examples and
provided some necessary conditions for simplicity. In \cite{CJO}, it
has been proven that there is an abundance of finite simple left
braces, indeed, for any positive integer $n>1$ and distinct prime
numbers  $p_1$, $p_2, \ldots , p_n$, there exist positive integers
$l_1,l_2,\ldots ,l_n$, such that, for each $n$-tuple of integers
$m_1 \geq l_1, m_2 \geq l_2,\dots ,m_n \geq l_n$, there exists a
simple left brace of order $p^{m_1} p^{m_2} \cdots  p^{m_n}$ that
has a metabelian multiplicative group with abelian Sylow subgroups.
The construction of these simple braces is via asymmetric products,
as introduced by Catino,  Colazzo and Stefanelli in \cite{CCS}. This
not only provided  constructions of new classes of simple left
braces, but also all previously known constructions have been
interpreted as asymmetric products. Furthermore, in \cite{BCJO19}, a
construction is given of finite simple left braces with a
multiplicative group $(B,\cdot )$ that is solvable of arbitrary
derived length. In this paper, we focus on the additive  group
$(B,+)$ and we discover new examples of finite simple left braces.
Our main result reads as follows:

{\it  For every finite abelian group $A$ there exists a finite
simple left brace $B$ such that $A$ is a subgroup of the additive
group of $B$.}

In fact, the simple left braces $B$ that we construct have
metabelian multiplicative groups $(B, \cdot)$ and have  abelian Sylow
subgroups (so called  A-groups \cite{Taunt}), and additive group isomorphic to
$$\prod_{i\in \Z/(m)}(\Z/(p_i^{n_i}))^{2s_il_{i-1}+1},$$
where $m>1$, $p_1,\dots p_m$ are any distinct prime numbers,
$n_1,\dots, n_m, s_1,\dots ,s_m$ are any positive integers, and
$l_i=p_i^{n_i}-p_i^{n_i-1}$. In particular, we thus also construct
examples that have elements of additive order $2^n$ (for any $n>1$).
None of the known previous constructions included such elements.

\section{Preliminaries}

A left brace is a set $B$ with two binary operations, $+$ and
$\cdot$, such that $(B,+)$ is an abelian group, $(B,\cdot)$ is a
group, and for every $a,b,c\in B$,
$$ a\cdot (b+c)+a=a\cdot b+a\cdot c.$$
In any left brace $B$ there is an action $\lambda\colon
(B,\cdot)\rightarrow \aut(B,+)$,  called the lambda map of $B$,
defined by $\lambda(a)=\lambda_a$ and $\lambda_{a}(b)=a\cdot b-a$,
for $a,b\in B$. A trivial brace is a left brace $B$ such that
$ab=a+b$, for all $a,b\in B$, i.e. all $\lambda_a=\id$.

A left ideal of a left brace $B$ is a subgroup $L$ of the additive
group of $B$ such that $\lambda_a(b)\in L$, for all $b\in L$ and all
$a\in B$. An ideal of a left brace $B$ is a normal subgroup $I$ of
the multiplicative group of $B$ such that $\lambda_a(b)\in I$, for
all $b\in I$ and all $a\in B$. Note that
$ab^{-1}=a-\lambda_{ab^{-1}}(b)$ and
$a-b=a\lambda_{a^{-1}b}(b^{-1})$, for all $a,b\in B$. Hence, every
left ideal $L$ of $B$ also is a subgroup of the multiplicative group
of $B$, and every  ideal $I$ of a left brace $B$ also is a subgroup
of the additive group of $B$, and then $B/I$ is a left brace, the
quotient brace $B$ modulo $I$. A non-zero left brace $B$ is simple
if $\{ 0\}$ and $B$ are the only ideals of $B$.

The following result is useful to prove simplicity in many cases.

\begin{lemma}\label{ideal} (Lemma 2.5 in \cite{BCJO18}) If $I$ is an ideal of a
left brace $B$, then $(\lambda_b-\id)(a)\in I$, for all $a\in B$ and
$b\in I$.
\end{lemma}

In \cite{CCS} Catino, Colazzo and Stefanelli introduced the
asymmetric product of two left braces. Let $S$ and $T$ be two
(additive) abelian groups. Recall that a (normalized) symmetric
$2$-cocycle on $T$ with values in $S$ is a map $b\colon T\times
T\rightarrow S$ such that
\begin{itemize}
\item[(i)] $b(0,0)=0$;
\item[(ii)] $b(t_1,t_2)=b(t_2,t_1)$;
\item[(iii)] $b(t_1+t_2,t_3)+b(t_1,t_2)=b(t_1,t_2+t_3)+b(t_2,t_3)$,
\end{itemize}
for all $t_1,t_2,t_3\in T$.

\begin{theorem} (Catino, Colazzo, Stefanelli \cite{CCS})
Let $T$ and $S$ be two left braces. Let $b \colon T\times T
\longrightarrow S$ be a symmetric $2$-cocycle on $(T, +)$ with
values in $(S, +)$, and let $\alpha \colon (S,\cdot)\longrightarrow
\aut(T,+,\cdot)$ be a homomorphism of groups such that
$$s\cdot b(t_2, t_3) + b(t_1\cdot \alpha_s(t_2 +
t_3),t_1)= b(t_1\cdot\alpha_s(t_2), t_1\cdot\alpha_s(t_3))+ s,$$
where $\alpha_s=\alpha(s)$, for all $s\in S$ and $t_1, t_2, t_3\in
T$. Then the addition and multiplication on $T\times S$ given by
$$
(t_1,s_1)+(t_2,s_2)=(t_1+t_2,~s_1+s_2+b(t_1,t_2)),
$$
$$
(t_1,s_1)\cdot (t_2,s_2)=(t_1\cdot \alpha_{s_1}(t_2),~s_1\cdot s_2),
$$
define a structure of left brace on $T\times S$, called the
asymmetric product of $T$ by $S$ (via $b$ and $\alpha$). It is
denoted as  $T\rtimes_\circ S$.
\end{theorem}

Note that, in $T\rtimes_\circ S$,
$$
(t_1,s_1)-(t_2,s_2)=(t_1-t_2,~s_1-s_2-b(t_1-t_2,t_2)),
$$
for all $t_1,t_2\in T$ and $s_1,s_2\in S$. Thus the lambda map of
$T\rtimes_\circ S$ is defined by
\begin{equation} \label{lambda} \lambda_{(t_1,s_1)}(t_2,s_2)= \left(
\lambda_{t_1}\alpha_{s_1}(t_2),~
\lambda_{s_1}(s_2)-b(\lambda_{t_1}\alpha_{s_1}(t_2),t_1)
\right).\end{equation}

Note that every symmetric bi-additive map $b\colon T\times
T\longrightarrow S$ is a symmetric $2$-cocycle on $(T, +)$ with
values in $(S, +)$. In this case, the condition on $b$ and  $\alpha$
is equivalent to the following two conditions:
$$\lambda_s(b(t_2,t_3))=b(\alpha_s(t_2),\alpha_s(t_3)),$$
$$b(t_2,t_3)=b(\lambda_{t_1}(t_2),\lambda_{t_1}(t_3)).$$
\bigskip

In the construction below $T$ and $S$ will be trivial braces. Hence,
these two conditions are reduced to
\begin{equation}\label{condition}
b(t_2,t_3)=b(\alpha_s(t_2),\alpha_s(t_3)).
\end{equation}
Moreover, in this case
$$\lambda_{(t_1,s_1)}(t_2,s_2)=
\left( \alpha_{s_1}(t_2),~ s_2-b(\alpha_{s_1}(t_2),t_1) \right).$$

Let $A,B$ be abelian additive groups. Let $b\colon A\times
A\longrightarrow B$ be a symmetric bi-additive map. We define an
addition on $A\times B$ by
$$(x_1,y_1)+(x_2,y_2)=(x_1+x_2,y_1+y_2+b(x_1,x_2)),
$$
for all $x_1,x_2\in A$ and $y_1,y_2\in B$. We denote by $G_b(A,B)$
the abelian group $A\times B$ with the above addition. Note that if
$b$ is the zero map $0$, then $G_0(A,B)$ is just the direct product
of the groups $A$ and $B$.

\begin{lemma}\label{additive}
If $B$ is finite of odd order, then $G_b(A,B)\cong G_0(A,B)$ for
every symmetric bi-additive map $b\colon A\times A\longrightarrow
B$.
\end{lemma}

\begin{proof}
Let $|B|=2n+1$. We define $\varphi_b\colon G_b(A,B)\longrightarrow
G_0(A,B)$ by $\varphi_b(x,y)=(x,y+nb(x,x))$, for all $x\in A$ and
$y\in B$. It is clear that $\varphi$ is bijective. For $x_1,x_2\in
A$ and $y_1,y_2\in B$, we have that
\begin{align*}
\varphi_b((x_1,y_1)+(x_2,y_2))=&\varphi_b(x_1+x_2,y_1+y_2+b(x_1,x_2))\\
=&(x_1+x_2,y_1+y_2+b(x_1,x_2)+nb(x_1+x_2,x_1+x_2))\\
=&(x_1+x_2,y_1+y_2+b(x_1,x_2)+nb(x_1,x_1)+nb(x_2,x_2)+2nb(x_1,x_2))\\
=&(x_1+x_2,y_1+y_2+nb(x_1,x_1)+nb(x_2,x_2))\\
=&\varphi_b(x_1,y_1)+\varphi_b(x_2,y_2).
\end{align*}
Thus the result follows.
\end{proof}

The result is not true for $B$ of even order. For example, let
$A=B=\Z/(2)$ and let $b\colon A\times A\longrightarrow B$ be the
map defined by $b(x,y)=xy$. Note that in $G_b(A,B)$ we have that
$(1,0)+(1,0)=(0,1)$. Then $G_b(A,B)\cong\Z/(4)\not\cong
\Z/(2)\times \Z/(2)=G_0(A,B)$.

\section{The construction}

Let $m>1$ be an integer. For $i\in \Z/(m)$, let $p_{i}$ be a prime
number and let $n_i$ and $s_i$ be positive integers. We assume that
$p_i\neq p_j$ for $i\neq j$. Consider the rings
$R_i=\mathbb{Z}/(p_i^{n_i})$ and the polynomials
$$q_i(x)=\sum_{k=0}^{p_i-1}x^{kp_i^{n_i-1}}\in R_{i+1}[x],$$
for $i\in \Z/(m)$. Note that
$(x^{p_i^{n_i-1}}-1)q_i(x)=x^{p_i^{n_i}}-1$. Put
$l_i=\deg(q_i(x))$. Let $C_i\in M_{l_{i-1}}(R_i)$ be the companion
matrix of $q_{i-1}(x)$. Note that $C_i$ is invertible in
$M_{l_{i-1}}(R_i)$ and has order $p_{i-1}^{n_{i-1}}$. Let
$T_i=R_i^{2l_{i-1}}$. We define the symmetric bilinear form
$b_i\colon T_i\times T_i\longrightarrow R_i$ by
$$b_i(u,v)=u\left(\begin{array}{c|c}
0&I_{l_{i-1}}\\  \hline I_{l_{i-1}}&0\end{array}\right)v^t,$$ for
all $u,v\in T_i$. We also define $f_i\in \aut(T_i)$ by
$$f_i(u)=u\left(\begin{array}{c|c}
C_i^t&0\\
\hline 0&C_i^{-1}\end{array}\right),$$ for all $u\in T_i$. One can
check that $f_i\in O(T_i,b_i)$, the orthogonal group of $b_i$. This
is a consequence of the fact that
$$\left(\begin{array}{c|c}
C_i^t&0\\
\hline 0&C_i^{-1}\end{array}\right)  \left(\begin{array}{c|c} 0&I_{l_{i-1}}\\
\hline I_{l_{i-1}}&0\end{array}\right) \left(\begin{array}{c|c}
C_i^t&0\\
\hline 0&C_i^{-1}\end{array}\right)^{t}= \left(\begin{array}{c|c} 0&I_{l_{i-1}}\\
\hline I_{l_{i-1}}&0\end{array}\right).$$

For every ring $R$, by the trivial brace $R$ will mean $(R,+,+)$.

Consider the trivial braces $T=T_1^{s_1}\times\cdots\times
T_m^{s_m}$ and $S=R_1\times\cdots\times R_m$. We define the
symmetric bi-additive map $b\colon T\times T\longrightarrow S$ by
\begin{align*}&b((u_{1,1},\dots,u_{1,s_1},\dots, u_{m,1},\dots,
u_{m,s_m}),(v_{1,1},\dots,v_{1,s_1},\dots, v_{m,1},\dots,
v_{m,s_m}))\\
&\qquad =(\sum_{j=1}^{s_1}b_1(u_{1,j},v_{1,j}),\dots
,\sum_{j=1}^{s_m}b_m(u_{m,j},v_{m,j})),
\end{align*}
for all $u_{i,j},v_{i,j}\in T_i$. We define $\alpha\colon
S\longrightarrow \aut(T)$ by $\alpha(a_1,\dots ,
a_m)=\alpha_{(a_1,\dots ,a_m)}$ and
\begin{eqnarray*}\lefteqn{\alpha_{(a_1,\dots ,a_m)}(u_{1,1},\dots,u_{1,s_1},\dots,
u_{m,1},\dots,
u_{m,s_m})}\\
&&=(f_1^{a_m}(u_{1,1}),\dots,f_1^{a_m}(u_{1,s_1}),f_2^{a_1}(u_{2,1})\dots,
f_2^{a_1}(u_{2,s_2}),\dots, f_m^{a_{m-1}}(u_{m,1}),\dots,
f_m^{a_{m-1}}(u_{m,s_m})),
\end{eqnarray*}
for all $a_i\in R_i$, and all $u_{i,j}\in T_i$. Note that, since
$f_i$ has order $p_{i-1}^{n_{i-1}}$ and $a_{i-1}\in
\Z/(p_{i-1}^{n_{i-1}})$, the map $\alpha$ is well-defined. Clearly
$\alpha$ is a group homomorphism from the group $(S,+)$ to
$\Aut(T,+,+)$. Since $f_i\in O(T_i,b_i)$, we have that
\begin{eqnarray*}\lefteqn{b(\alpha_{(a_1,\dots, a_m)}(u_{1,1},\dots,u_{1,s_1},\dots,
u_{m,1},\dots, u_{m,s_m}),\alpha_{(a_1,\dots
,a_m)}(v_{1,1},\dots,v_{1,s_1},\dots, v_{m,1},\dots,
v_{m,s_m}))}\\
&&=b((f_1^{a_m}(u_{1,1}),\dots,f_1^{a_m}(u_{1,s_1}),f_2^{a_1}(u_{2,1})\dots,
f_2^{a_1}(u_{2,s_2}),\dots, f_m^{a_{m-1}}(u_{m,1}),\dots,
f_m^{a_{m-1}}(u_{m,s_m})),\\
&&\qquad
(f_1^{a_m}(v_{1,1}),\dots,f_1^{a_m}(v_{1,s_1}),f_2^{a_1}(v_{2,1})\dots,
f_2^{a_1}(v_{2,s_2}),\dots, f_m^{a_{m-1}}(v_{m,1}),\dots,
f_m^{a_{m-1}}(v_{m,s_m}))),\\
&&=(\sum_{j=1}^{s_1}b_1(f_1^{a_m}(u_{1,j}),f_1^{a_m}(v_{1,j})),\dots
,\sum_{j=1}^{s_m}b_m(f_{m}^{a_{m-1}}(u_{m,j}),f_{m}^{a_{m-1}}(v_{m,j}))),\\
&& =(\sum_{j=1}^{s_1}b_1(u_{1,j},v_{1,j}),\dots
,\sum_{j=1}^{s_m}b_m(u_{m,j},v_{m,j})),\\
&&=b((u_{1,1},\dots,u_{1,s_1},\dots, u_{m,1},\dots,
u_{m,s_m}),(v_{1,1},\dots,v_{1,s_1},\dots, v_{m,1},\dots,
v_{m,s_m})),
\end{eqnarray*}
for all $a_i\in R_i$ and all $u_{i,j},v_{i,j}\in T_i$.

Hence, (\ref{condition}) is satisfied and we can construct the
asymmetric product $T\rtimes_{\circ}S$ of $T$ by $S$ via $\alpha$
and $b$.

\begin{lemma}\label{addSylow} The Sylow $p_i$-subgroup of the additive group of
$T\rtimes_{\circ}S$ is
$$A_i=\{0\}\times\cdots\times\{0\}\times T_i^{s_i}\times\{0\}\times\cdots\times\{0\}\times R_i\times \{0\}\times\cdots\times \{0\}.$$
Furthermore $A_i\cong (\Z/(p_i^{n_i}))^{2s_il_{i-1}+1}$.
\end{lemma}
\begin{proof}   Note that
$|T\rtimes_{\circ}S|=\prod_{j\in\Z/(m)}p_j^{n_j(2s_jl_{j-1}+1)}$
and $|A_i|=p_i^{n_i(2s_il_{i-1}+1)}$. Thus to prove that $A_i$ is
the Sylow $p_i$-subgroup of the additive subgroup of
$T\rtimes_{\circ}S$, it is enough to show that $A_i$ is an
additive subgroup of $T\rtimes_{\circ}S$. But this is obvious from
the definition of the addition in $T\rtimes_{\circ}S$. Therefore
the first part of the result follows.

Let $u,v\in A_i$. Then
$$u=((0,\dots, 0,u_{i,1},\dots
,u_{i,s_i},0,\dots,0),(0,\dots ,0,a_i,0,\dots ,0))\in A_i$$ and
$$v=((0,\dots, 0,v_{i,1},\dots ,v_{i,s_i},0,\dots,0),
(0,\dots ,0,c_i,0,\dots ,0))\in A_i,$$ for some $u_{i,k},v_{i,k}\in
T_i$ and $a_i,c_i\in R_i$. Since $T_i=R_i^{2l_{i-1}}$, there exist
$a_{i,k,j},c_{i,k,j}\in R_i$ such that
$$u_{i,k}=(a_{i,k,1},\dots ,a_{i,k,2l_{i-1}})\quad\mbox{and}\quad v_{i,k}=(c_{i,k,1},\dots ,c_{i,k,2l_{i-1}}).$$
We define $\varphi\colon A_i\longrightarrow
(\Z/(p_i^{n_i}))^{2s_il_{i-1}+1}$ by
\begin{align*}\varphi(u)&=(a_{i,1,1},\dots, a_{i,1,2l_{i-1}},a_{i,2,1},\dots ,a_{i,2,2l_{i-1}},\dots ,a_{i,s_i,1},\dots
,a_{i,s_i,2l_{i-1}},\\
&\quad
a_i-\sum_{j=1}^{s_i}\sum_{k=1}^{l_{i-1}}a_{i,j,k}a_{i,j,k+l_{i-1}}).
\end{align*}
Note that
\begin{align*}
u+v&=((0,\dots, 0,u_{i,1}+v_{i,1},\dots
,u_{i,s_i}+v_{i,s_i},0,\dots,0),(0,\dots,
0,a_i+c_i+\sum_{j=1}^{s_i}b_i(u_{i,j},v_{i,j}),0\dots,0))\\
&=((0,\dots, 0,u_{i,1}+v_{i,1},\dots
,u_{i,s_i}+v_{i,s_i},0,\dots,0),\\
&\qquad (0,\dots,
0,a_i+c_i+\sum_{j=1}^{s_i}\sum_{k=1}^{l_{i-1}}(a_{i,j,k}c_{i,j,k+l_{i-1}}+a_{i,j,k+l_{i-1}}c_{i,j,k}),0\dots,0))
\end{align*}
and then
\begin{align*}
\varphi(u+v)&=(a_{i,1,1}+c_{i,1,1},\dots,
a_{i,1,2l_{i-1}}+c_{i,1,2l_{i-1}},a_{i,2,1}+c_{i,2,1},\dots
,a_{i,2,2l_{i-1}}+c_{i,2,2l_{i-1}},\\
&\quad \dots ,a_{i,s_i,1}+c_{i,s_i,1},\dots
,a_{i,s_i,2l_{i-1}}+c_{i,s_i,2l_{i-1}},\\
&\quad
a_i+c_i+\sum_{j=1}^{s_i}\sum_{k=1}^{l_{i-1}}(a_{i,j,k}c_{i,j,k+l_{i-1}}+a_{i,j,k+l_{i-1}}c_{i,j,k})\\
&\qquad-\sum_{j=1}^{s_i}\sum_{k=1}^{l_{i-1}}(a_{i,j,k}+c_{i,j,k})(a_{i,j,k+l_{i-1}}+c_{i,j,k+l_{i-1}}))\\
&=(a_{i,1,1}+c_{i,1,1},\dots,
a_{i,1,2l_{i-1}}+c_{i,1,2l_{i-1}},a_{i,2,1}+c_{i,2,1},\dots
,a_{i,2,2l_{i-1}}+c_{i,2,2l_{i-1}},\\
&\quad \dots ,a_{i,s_i,1}+c_{i,s_i,1},\dots
,a_{i,s_i,2l_{i-1}}+c_{i,s_i,2l_{i-1}},\\
&\quad
a_i+c_i-\sum_{j=1}^{s_i}\sum_{k=1}^{l_{i-1}}(a_{i,j,k}a_{i,j,k+l_{i-1}}+c_{i,j,k+l_{i-1}}c_{i,j,k})\\
&=\varphi(u)+\varphi(v).
\end{align*}
Hence $\varphi$ is a homomorphism of groups. Clearly $\varphi$ is
bijective. Therefore the result follows.
\end{proof}

\begin{remark}\label{multSylow}
{\rm Note that the multiplicative group of the left brace
$T\rtimes_{\circ}S$ constructed above is the semidirect product
$T\rtimes_{\alpha}S$ of $T$ and $S$ via $\alpha$. Hence, it is
metabelian. Since $A_i$ is the additive Sylow $p_i$-subgroup of the
left brace $T\rtimes_{\circ}S$, it is clear that $A_i$ is a left
ideal. Thus the multiplicative group of $A_i$ is a multiplicative
Sylow $p_i$-subgroup of $T\rtimes_{\alpha}S$. By the definition of
$\alpha$, it is easy to see that the multiplicative group of $A_i$
is isomorphic to the direct product $T_{i}^{s_i}\times
R_i=(\Z/(p_i^{n_i}))^{2s_il_{i-1}+1}$. Hence the Sylow subgroups of
the multiplicative group of the left brace $T\rtimes_{\circ}S$ are
abelian, i.e. $T\rtimes_{\circ}S$ is an A-group.}
\end{remark}

\begin{proposition}\label{main}
With the above notation, the left brace $T\rtimes_{\circ}S$ is
simple.
\end{proposition}
\begin{proof}
Let $I$ be a non-zero ideal of $T\rtimes_{\circ}S$. Let
$((u_{1,1},\dots,u_{1,s_1},\dots, u_{m,1},\dots,
u_{m,s_m}),(a_1,\dots ,a_m))\in I$ be a non-zero element. By
symmetry, we may assume that $(u_{1,1},\dots,u_{1,s_1},a_1)\neq
(0,\dots ,0)$. Since $p_1,\dots, p_m$ are different prime numbers,
we also may assume that $(u_{i,1},\dots,u_{i,s_1},a_i)=(0,\dots
,0)$, for all $i\neq 1$. Thus $((u_{1,1},\dots,u_{1,s_1},0\dots
0),(a_1,0,\dots ,0))\in I$.

We shall prove that $((0,\dots ,0),(p_1^{n_1-1},0,\dots ,0))\in I$.

Note that if $(u_{1,1},\dots,u_{1,s_1})=(0,\dots ,0)$, then $a_1\neq
0$ and, in this case, $((0,\dots ,0),(p_1^{n_1-1},0,\dots ,0))\in
I$, because every nontrivial additive subgroup of
$\mathbb{Z}/(p_i^{n_i})$ contains $p_i^{n_i-1}$.

Suppose that $(u_{1,1},\dots,u_{1,s_1})\neq (0,\dots ,0)$. In this
case, since $b_1$ is non-singular, there exists
$(v_{1,1},\dots,v_{1,s_1})\in T_1^{s_1}$ such that
$\sum_{j=1}^{s_1}b_1(u_{1,j},v_{1,j})\neq 0$. Now, since $I$ is
closed by the lambda maps, we get
\begin{eqnarray*}\lefteqn{\lambda_{((v_{1,1},\dots,v_{1,s_1},0,\dots ,0),(0,\dots
,0))}((u_{1,1},\dots,u_{1,s_1},0,\dots ,0), (a_1,0,\dots
,0))}\\
&&=((u_{1,1},\dots,u_{1,s_1},0,\dots
,0),(a_1-\sum_{j=1}^{s_1}b_1(u_{1,j},v_{1,j}),0, \dots ,0))\in
I.\end{eqnarray*}
Hence
\begin{eqnarray*}\lefteqn{((u_{1,1},\dots,u_{1,s_1},0,\dots ,0),(a_1,0,\dots
,0))}\\
&&-((u_{1,1},\dots,u_{1,s_1},0,\dots
,0),(a_1-\sum_{j=1}^{s_1}b_1(u_{1,j},v_{1,j}),0,\dots
,0))\\
&=&((0,\dots ,0),(\sum_{j=1}^{s_1}b_1(u_{1,j},v_{1,j}),0,\dots
,0))\in I,\end{eqnarray*} and thus also $((0,\dots
,0),(p_1^{n_1-1},0,\dots ,0))\in I$ in this case.

Now, by Lemma~\ref{ideal},  we have that
\begin{eqnarray} \label{inI} \lefteqn{(\lambda_{((0,\dots ,0),(p_1^{n_1-1},0,\dots ,0))}-\id)((0\dots
,0,v_{2,1},\dots ,v_{2,s_2},0,\dots ,0),(0,\dots ,0))}\nonumber\\
&&=((0,\dots ,0,(f_2^{p_1^{n_1-1}}-\id)(v_{2,1}),\dots
,(f_2^{p_1^{n_1-1}}-\id)(v_{2,s_2}),0,\dots ,0),\nonumber\\
&&\qquad
(0,-\sum_{j=1}^{s_2}b_2(v_{2,j},(f_2^{p_1^{n_1-1}}-\id)(v_{2,j})),0,\dots
,0))\in I,
\end{eqnarray}
for every $(v_{2,1},\dots ,v_{2,s_2})\in T_2^{s_2}$. Note  that
$$q_1(x)=p_1+(x^{p_1^{n_1-1}}-1)\sum_{j=1}^{p_1-1}jx^{(p_1-j-1)p_1^{n_1-1}}.$$
Hence
$$0=q_1(C_2)=p_1I_{l_2}+(C_2^{p_1^{n_1-1}}-I_{l_2})\sum_{j=1}^{p_1-1}jC_2^{(p_1-j-1)p_1^{n_1-1}}.$$
Since $p_1\in R_2$ is invertible,  this implies that
$f_2^{p_1^{n_1-1}}-\id$ is invertible. Hence, by the definition of
$b_2$, there exist $v_{2,j},v'_{2,j}\in T_2$ such that
\begin{equation}\label{invertible}\sum_{j=1}^{s_2}b_2((f_2^{p_1^{n_1-1}}-\id)(v_{2,j}),v'_{2,j})=1.\end{equation}
Indeed, there exists $v_{2,1}\in T_2$ such that
$(f_2^{p_1^{n_1-1}}-\id)(v_{2,1})=(1,0,\dots ,0)$. Let
$v'_{2,1}=(1,\dots ,1)$ and $v'_{2,j}=(0,\dots ,0)$, for all
$1<j\leq s_2$. Then $b_2((1,0,\dots ,0),(1,\dots ,1))=1$ and thus
we get (\ref{invertible}). By (\ref{inI})
\begin{eqnarray}\label{vinI}\lefteqn{((0,\dots
,0,(f_2^{p_1^{n_1-1}}-\id)(v_{2,1}),\dots
,(f_2^{p_1^{n_1-1}}-\id)(v_{2,s_2}),0,\dots
,0),}\nonumber\\
&&\qquad
(0,-\sum_{j=1}^{s_2}b_2(v_{2,j},(f_2^{p_1^{n_1-1}}-\id)(v_{2,j})),0,\dots
,0))\in I.
\end{eqnarray}
Then, since $I$ is closed by the lambda maps, applying
$\lambda_{((0,\dots ,0,v'_{2,1},\dots ,v'_{2,s_2},0,\dots
,0),(0,\dots ,0))}$ to the element in (\ref{vinI}), we have that
\begin{eqnarray*}\lefteqn{\lambda_{((0,\dots ,0,v'_{2,1},\dots ,v'_{2,s_2},0,\dots
,0),(0,\dots ,0))}((0,\dots
,0,(f_2^{p_1^{n_1-1}}-\id)(v_{2,1}),\dots
,(f_2^{p_1^{n_1-1}}-\id)(v_{2,s_2}),0,\dots
,0),}\\
&&\qquad
(0,-\sum_{j=1}^{s_2}b_2(v_{2,j},(f_2^{p_1^{n_1-1}}-\id)(v_{2,j})),0,\dots
,0))\\
&&=((0,\dots ,0,(f_2^{p_1^{n_1-1}}-\id)(v_{2,1}),\dots
,(f_2^{p_1^{n_1-1}}-\id)(v_{2,s_2}),0,\dots ,0),\\
&&\qquad
(0,-\sum_{j=1}^{s_2}b_2(v_{2,j},(f_2^{p_1^{n_1-1}}-\id)(v_{2,j}))-\sum_{j=1}^{s_2}b_2((f_2^{p_1^{n_1-1}}-\id)(v_{2,j}),v'_{2,j}),0,\dots
,0))\\
&&=((0,\dots ,0,(f_2^{p_1^{n_1-1}}-\id)(v_{2,1}),\dots
,(f_2^{p_1^{n_1-1}}-\id)(v_{2,s_2}),0,\dots ,0),\\
&&\qquad
(0,-\sum_{j=1}^{s_2}b_2(v_{2,j},(f_2^{p_1^{n_1-1}}-\id)(v_{2,j}))-1,0,\dots
,0))\in I,
\end{eqnarray*}
where the last equality follows by (\ref{invertible}).  Hence, this
together with (\ref{vinI}) yields
\begin{eqnarray}\label{1inI}\lefteqn{((0,\dots ,0,(f_2^{p_1^{n_1-1}}-\id)(v_{2,1}),\dots ,(f_2^{p_1^{n_1-1}}-\id)(v_{2,s_2}),0,\dots
,0),}\nonumber\\
&&\qquad (0,-\sum_{j=1}^{s_2}b_2(v_{2,j},(f_2^{p_1^{n_1-1}}-\id)(v_{2,j})),0\dots ,0))\nonumber\\
&&-((0,\dots ,0,(f_2^{p_1^{n_1-1}}-\id)(v_{2,1}),\dots
,(f_2^{p_1^{n_1-1}}-\id)(v_{2,s_2}),0,\dots
,0),\nonumber\\
&&\qquad
(0,-\sum_{j=1}^{s_2}b_2(v_{2,j},(f_2^{p_1^{n_1-1}}-\id)(v_{2,j}))-1,0\dots
,0))\nonumber\\
&=&((0,\dots ,0),(0,1,0,\dots ,0))\in I.
\end{eqnarray}
Now we have that
$$((0,\dots ,0),(0,\sum_{j=1}^{s_2}b_2(w_{2,j},(f_2^{p_1^{n_1-1}}-\id)(w_{2,j})),0,\dots ,0))\in I,$$
for all $(w_{2,1},\dots ,w_{2,s_2})\in T_2^{s_2}$, and by
(\ref{inI}),  we also have that
\begin{eqnarray*}\lefteqn{((0,\dots ,0,(f_2^{p_1^{n_1-1}}-\id)(w_{2,1}),\dots
,(f_2^{p_1^{n_1-1}}-\id)(w_{2,s_2}),0,\dots ,0),}\\
&&\qquad (0,-\sum_{j=1}^{s_2}b_2(w_{2,j},(f_2^{p_1^{n_1-1}}-\id)(w_{2,j})),0,\dots ,0))\\
&&\quad+((0,\dots
,0),(0,\sum_{j=1}^{s_2}b_2(w_{2,j},(f_2^{p_1^{n_1-1}}-\id)(w_{2,j})),0,\dots
,0))\in I,
\end{eqnarray*}
for all $(w_{2,1},\dots ,w_{2,s_2})\in T_2^{s_2}$. Thus
\begin{eqnarray*}\lefteqn{((0,\dots ,0,(f_2^{p_1^{n_1-1}}-\id)(w_{2,1}),\dots
,(f_2^{p_1^{n_1-1}}-\id)(w_{2,s_2}),0,\dots ,0),}\\
&&\qquad (0,-\sum_{j=1}^{s_2}b_2(w_{2,j},(f_2^{p_1^{n_1-1}}-\id)(w_{2,j})),0,\dots ,0))\\
&&\quad+((0,\dots
,0),(0,\sum_{j=1}^{s_2}b_2(w_{2,j},(f_2^{p_1^{n_1-1}}-\id)(w_{2,j})),0,\dots
,0))\\
&&=((0,\dots ,0,(f_2^{p_1^{n_1-1}}-\id)(w_{2,1}),\dots
,(f_2^{p_1^{n_1-1}}-\id)(w_{2,s_2}),0,\dots ,0),(0,\dots,0))\in I,
\end{eqnarray*}
for all $(w_{2,1},\dots ,w_{2,s_2})\in T_2^{s_2}$. Since
$f_2^{p_1^{n_1-1}}-\id$ is invertible, this yields that
$$((0,\dots
,0,t_{2,1},\dots ,t_{2,s_2},0,\dots ,0),(0,\dots ,0))\in I,$$ for
all $(t_{2,1},\dots ,t_{2,s_2})\in T_2^{s_2}$. Since,  by
(\ref{1inI}), $((0,\dots ,0),(0,1,0,\dots ,0))\in I$, we have that
$$((0,\dots
,0,t_{2,1},\dots ,t_{2,s_2},0,\dots ,0),(0,d_2,0\dots ,0))\in I,$$
for all $(t_{2,1},\dots ,t_{2,s_2},d_2)\in T_2^{s_2}\times R_2$.
Now, by a similar argument, one can prove that
$$((0,\dots ,0,t_{i,1},\dots
,t_{i,s_i},0,\dots ,0)(0,\dots ,0,d_i,0,\dots ,0))\in I,$$ for all
$(t_{i,1},\dots ,t_{i,s_i},d_i)\in T_i^{s_i}\times R_i$, for all
$i\in \Z/(m)$, and the result follows.
\end{proof}
Hence, as an easy consequence of Proposition~\ref{main},
Lemma~\ref{addSylow} and Remark~\ref{multSylow},  we get the desired
result.
\begin{theorem}
Let $A$ be a finite abelian group. Then $A$ is a subgroup of the
additive group of a finite simple left brace $B$ with metabelian
multiplicative group with abelian Sylow subgroups.
\end{theorem}

\vspace{30pt}
 \noindent \begin{tabular}{llllllll}
  F. Ced\'o && E. Jespers\\
 Departament de Matem\`atiques &&  Department of Mathematics \\
 Universitat Aut\`onoma de Barcelona &&  Vrije Universiteit Brussel \\
08193 Bellaterra (Barcelona), Spain    && Pleinlaan 2, 1050 Brussel, Belgium \\
 cedo@mat.uab.cat && Eric.Jespers@vub.be \\ \\
 J. Okni\'{n}ski && \\ Institute of
Mathematics &&
\\  Warsaw University &&\\
 Banacha 2, 02-097 Warsaw, Poland &&\\
 okninski@mimuw.edu.pl &&
\end{tabular}


\begin{thebibliography}{99}
\itemsep=-2pt
\bibitem{B18} D. Bachiller, Extensions, matched products and simple
braces, J.  Pure  Appl. Algebra 222 (2018), 1670--1691.
\bibitem{BCJAll} D. Bachiller, F. Ced\'o and E. Jespers, Solutions of the Yang-Baxter
equation associated with a left brace, J. Algebra 463 (2016), 80--102.
\bibitem{BCJO18} D. Bachiller, F. Ced\'o, E. Jespers and J.
Okni\'{n}ski, Iterated matched products of finite braces and
simplicity; new solutions of the Yang-Baxter equation,  Trans. Amer.
Math. Soc. 370 (2018), 4881--4907.
\bibitem{BCJO19} D. Bachiller, F.
Ced\'o, E. Jespers and J. Okni\'{n}ski, Asymmetric product of left
braces and simplicity; new solutions of the Yang-Baxter equation,
Communications in Contemporary Mathematics 21 (2019), 1850042 (30
pages).
\bibitem{CCS} F. Catino, I. Colazzo and P. Stefanelli, Regular
subgroups of the afine group and asymmetric product of braces,  J.
Algebra 455 (2016), 164--182.
\bibitem{CedoSurvey} F. Ced\'o,  Left braces: solutions of the Yang--Baxter equation, Adv. Group Theory Appl. 5 (2018), 33--90.
\bibitem{CJOComm} F. Ced\'o, E. Jespers and J. Okni\'nski, Braces and the Yang--Baxter equation,
Commun. Math. Phys. 327 (2014), 101--116.
\bibitem{CJO} F. Ced\'o, E. Jespers
and J. Okni\'{n}ski, An abundance of simple left braces with abelian
multiplicative Sylow subgroups (to appear in Revista Matem\'{a}tica
Iberoamericana).
\bibitem{Drinfeld} V. G. Drinfeld, On unsolved problems in quantum group theory. Quantum
Groups, Lecture Notes Math. 1510, Springer-Verlag, Berlin, 1992, 1--8.
\bibitem{ESS} P. Etingof, T. Schedler and A.
Soloviev, Set-theoretical solutions to the quantum Yang-Baxter
equation, Duke Math. J. 100 (1999), 169--209.
\bibitem{R07} W. Rump,
Braces, radical rings, and the quantum Yang-Baxter equation, J.
Algebra 307 (2007), 153--170.
\bibitem{RumpSurvey} W. Rump,  The brace of a classical group, Note Mat. 34 (2014), no. 1, 115--144.
\bibitem{SmokRestraints} A. Smoktunowicz, A note on set-theoretic solutions of the Yang--Baxter equation, J.
Algebra 500 (2018), 3--18.
\bibitem{Taunt} D.R. Taunt,  On A-groups, Proc. Cambridge Philos. Soc. 45 (1949), 24--42.
\end{thebibliography}
\end{document}